\documentclass[12pt,reqno,twoside]{amsart}
\usepackage{amsmath, amssymb,latexsym}
\usepackage[all,cmtip]{xy}
\usepackage{enumitem}

\newtheorem{thm}{Theorem}[section]
\newtheorem{lem}[thm]{Lemma}
\newtheorem{prop}[thm]{Proposition}

\newtheorem{defn}[thm]{Definition}

\newtheorem{question}{Question}

\numberwithin{equation}{section}

\newcommand{\ran}{{\rm{ran}}}

\newcommand{\N}{{\mathbb{N}}}

\begin{document}

\author[Peter Cholak ]{Peter Cholak}

\address{\tt Department of Mathematics, University of Notre Dame, 255
  Hurley, Notre Dame, IN 46556, USA} \email{\tt Peter.Cholak.1@nd.edu}

\author[Rodney G. Downey]{Rodney G. Downey}

\address{\tt School of Mathematics, Statistics and Operations
  Research, Victoria University, P.O. Box 600, Wellington, New
  Zealand} \email{\tt Rod.Downey@vuw.ac.nz}

\author[Gregory Igusa ]{Greg Igusa}

\address{\tt Department of Mathematics, University of Notre Dame, 255
  Hurley, Notre Dame, IN 46556, USA} \email{\tt
  Gregory.Igusa.1@nd.edu}

\thanks{Downey was supported by the Marsden Fund of New Zealand.  This
  work was partially supported by a grant from the Simons Foundation
  (\#315283 to Peter Cholak).  Igusa was partially supported by
  EMSW21-RTG-0838506.  Research was (partially) completed while the
  authors were visiting the Institute for Mathematical Sciences,
  National University of Singapore in 2014}

\title[FIP]{Any FIP real computes a 1-generic}

\maketitle

\begin{abstract}

  We construct a computable sequence of computable reals
  $\langle X_i\rangle$ such that any real that can compute a
  subsequence that is maximal with respect to the finite intersection
  property can also compute a Cohen 1-generic.  This is extended to
  establish the same result with 2IP in place of FIP.  This is the
  first example of a classical theorem of mathematics has been found
  to be equivalent, both proof theoretically and in terms of its
  effective content, to computing a 1-generic.

\end{abstract}

\section{Introduction}
Of all the axioms of mathematics, perhaps the most well-known is the
Axiom of Choice; which has many classically equivalent forms. These
include, for example Zorn's Lemma, the Well-Ordering Principle, etc.
Rubin and Rubin \cite{RR} give many such classically equivalent
versions.

These equivalent forms do not, however, remain equivalent when we
beging to examine them through the fine grained lenses of either
computability theory or proof theory, especially via reverse
mathematics.  One of the classic versions of the axiom of choice is
called the \emph{finite intersection property.}  In the following
definition we will think of reals as being subsets of ${\mathbb N}$.
We will often abuse notation and refer to reals interchangeably as
sets or vice-versa.

\begin{defn}

  A set (or sequence) of reals has the finite intersection property if
  given any finite subset of the reals (or any finite set of reals
  that appear in the sequence) that set has a nonempty intersection.

\end{defn}

In \cite{DM12,DM13}, Dzhafarov and Mummert began a systematic study of
the proof-theoretical and computability-theoretical strength of
various incarnations of the axiom of choice and related principles of
finite character.  Of concern to the present paper, Dzhafarov and
Mummert initiated an investigation of such properties for the finite
intersection principle.  In this paper we will be concerned with a
logical analysis of this principle.

As with Dzhafarov and Mummert's analysis, it is quite important how we
present the family. An \emph{instance} of an intersection problem is a
sequence of reals, subsets of ${\mathbb N}$,
${\mathcal X}=\{A_i\mid i<\omega\}$.  We allow repetitions in the
family.  There are possibly many equivalent enumerations of the same
families and these may not be computationally equivalent.  We will
write that $\hat{{\mathcal X}}\subseteq {\mathcal X}$ if every element
of the sequence $\hat{{\mathcal X}}$ is in ${\mathcal X}.$

Thus we will consider the following.

\begin{defn}
  Let $\mathbb{X}=\langle X_i\mid i<\omega \rangle$ be a uniformly
  computable sequence of computable reals, not all empty. Let
  $f:\N\rightarrow\N$. Then, $f$ is FIP with respect to $\mathbb{X}$
  if $\{X_{f(i)}\mid i<\omega \}$ is maximal with respect to the
  finite intersection property, that is to say, $\{X_{f(i)}\}$ has the
  finite intersection property, and if
  $\{X_{f(i)}\}\subset \{X_{g(i)}\}$, then $\{X_{g(i)}\}$ does not
  have the finite intersection property.

  In this case, we will frequently abuse notation and say that
  $\langle X_{f(i)}\rangle$ or $\{ X_{f(i)}\}$ is FIP.

\end{defn}

A related principle is called $n$IP\footnote{This principle was
  referred to as $\overline{D}_n-IP$ by Dzhafarov and Mummert.  We
  follow the terminology of Downey, Diamondstone, Greenberg and
  Turetsky \cite{DDGT}.}, and it is the same as the above, except it
only seeks each \emph{$n$-tuple} of sets to have nonempty
intersection.  The principle 2IP thus says that every family of sets
has a maximal subfamily where each pair of sets in the family has
nonempty intersection.


In \cite{DM13}, Dzhafarov and Mummert showed that there was a
computable instance of FIP with no computable solution.

\begin{defn}[Dzhafarov and Mummert \cite{DM13}] A degree ${\bf a}$ is
  FIP (or FIP-bounding) if
  given any computable $\mathbb{X}$ as above, ${\bf a}$ can compute a
  function $f$ that is FIP with respect to $\mathbb{X}$.  Similarly we
  can define a degree to be $n$IP.
\end{defn}

It is easy to see that ${\bf 0'}$ is FIP and $n$IP, simply by taking a
computable family and making a maximal FIP subfamily one set at a
time, noting that nonempty intersection of computable sets is
$\Sigma_1^0$.  Dzhafarov and Mummert established each FIP degree is
also $n$Ip for all $n$, and any $(n+1)$IP degree is also $n$IP.  One
of the fundamental questions left open by Dzhafarov and Mummert was
whether any of these implications are proper.  One of the corollaries
of the work in the present paper is that the answer is, somewhat
surprisingly, no.

Dzhafarov and Mummert demonstrated that FIP has relationships with a
number of well-studied principles of Reverse Mathematics and
computability theory. These include cohesive sets, and the Atomic
Model Theorem from model theory.  Of relevance to us will be the
following which relates FIP to hyperimmune and 1-generic degrees.  The
reader should recall that a degree ${\bf a}$ is called
\emph{hyperimmune} if ${\bf a}$ computes a function $f$ which is not
dominated by any computable function: if $g$ is computable
$\exists^\infty s(f(s)>g(s)).$ Some authors refer to degrees which are
not hyperimmune as \emph{computably dominated} for this reason.

The reader should recall that ${\bf a}$ is called 1-generic if it
contains a 1-generic set $G$, where $G$ is 1-generic if it is Cohen
generic for 1-quantifier arithmetic. Equivalently, for all c.e. sets
of strings $V$, there is a string $\tau\prec G$ such that either
$\tau\in V$ or for all $\sigma\in V$, $\tau\not\preceq \sigma$.  Such
sets encode the basic behaviors coming to all finite extension
arguments.  All 1-generic sets are hyperimmune.

\begin{thm}[Dzhafarov and Mummert \cite{DM13}]
  Every FIP degree is hyperimmune. The following are FIP degrees.
  \begin{enumerate}
  \item All nonzero c.e. degrees.
  \item Degrees that are ${\bf 0'}$-hyperimmune.
  \item Degres that compute 1-generics which meet a prescribed
    sequence of dense $\Pi_1^0$ sets of strings.
  \end{enumerate}
\end{thm}

Dzhafarov and Mummert also partially classified how the FIP degrees
relate to other bounding degrees such as the atomic model bounding
degrees and prime model bounding degrees (see, e.g.  \cite{CHKS,HSS}).

In view of this, the precise classification of the FIP and $n$IP
degrees seems a very intriguing problem.  Recently, Diamondstone,
Downey, Greenberg and Turetsky \cite{DDGT} made some significant
advances for the classification, as well as establishing some basic
results. They showed that the FIP and 2IP degrees have acceptable
universal families, meaning that if a degree computes a solution to
the universal family, then it will compute a solution to any family.
They also proved that if $G$ is 1-generic then the Turing degree of
$G$ is 2IP and hence FIP.\footnote{Note that \emph{more} families have
  2IP than FIP. For example, the family consisting of the 3 sets
  $\{1,2\},\{2,3\} $ and $\{3,1\}$ has 2IP but not FIP.} Those authors
also looked at finite variations of FIP and $n$IP showing that the
principles differed in strength for some such variations, and also
looked at the situation where \emph{infinite} intersection was asked
for. In this last formulation, the precise strength needed was
found. (Namely ${\mathbf a}$ can compute a solution with infinite
intersection property for all computable families iff
${\bf a\ge 0''}).$

Diamondstone, Downey, Greenberg and Turetsky \cite{DDGT} gave a
complete answer to the question of FIP and 2IP bounding in a limited
setting.

\begin{thm}[Diamondstone, Downey, Greenberg and Turetsky \cite{DDGT}]
  A $\Delta_2^0$ degree is FIP iff it is 2IP iff it computes a
  1-generic.
\end{thm}

The goal of the present paper is to give a complete classification of
the FIP and $n$IP degrees.

\begin{thm} \label{main} A degree ${\bf a}$ is FIP iff it is 2IP iff
  it computes a 1-generic.
\end{thm}

In view of results in \cite{DDGT}, it is enough to prove that if
${\bf a}$ is an FIP (2IP) degree then it computes a 1-generic.  We
remark that Theorem \ref{main} gives the first example of a classical
theorem whose strength aligns exactly with the ability to construct a
1-generic.

In the proof, we will first establish the FIP case.  Later we will
modify this proof to establish the 2IP case. We remark that the proofs
from \cite{DDGT} heavily used approximations and were uniform in the
sense that a single procedure $\Phi$ was constructed which took the
given $\Delta_2$ solution $S$ to the 1-generic $G$. Preliminary
analyses show that no such uniform solution is possible for the global
case.  Our proof is nonuniform, and constructs infinitely many
possible functionals and we will argue that at least one works.

In the final section, we will discuss the proof theory of the
situation.  We will do this in the context of reverse mathematics in
the sense of Simpson \cite{Si}.  A general reference for computability
is Soare \cite{So}.

\section{The FIP case}\label{fip1}
In this section, we construct an $\mathbb{X}=\langle X_i\rangle$ such
that any FIP $f$ can compute a 1-generic. We will use one primary
Turing functional $\Phi$ to compute a 1-generic from $f$, and a
countable collection of secondary functionals $\Psi_i$ to compute the
halting set, $K$, from $f$. The Turing degree of $K$ is $\emptyset'$,
which can be used to compute a 1-generic. The outputs of the
functionals will depend only on the range of $f$. (Intuitively, this
means that the 1-generic will depend only on the FIP set, not on the
FIP sequence.) We will guarantee that for any FIP $f$, at least one of
these functionals succeeds at its task.

We produce a collection of labels that we will place onto the tree
$2^{<\omega}$. Each $\sigma$ in $2^{<\omega}$ will have either one
label, or an infinite countable collection of labels on it. Each label
will be a name for a computable real, and the sequence
$\langle X_i\rangle$ will be the sequence of labels that are used in
our construction, in the order that they first appear. The
construction will ensure that for any finite or infinite path through
$2^{<\omega}$, the sequence of reals whose labels appear on the path
has the finite intersection property (but is, perhaps, not maximal
with respect to having the finite intersection property).


The labels that we will use are as follows.

We will have labels of the form $A_{\sigma}$, thought of as structural
labels.  For each $\sigma\in 2^{<\omega}$, the label $A_\sigma$ will
be placed on $\sigma$ and in no other location. The purpose of these
labels will be to help tie $f$ to a path through $2^{<\omega}$. These
labels will be what $\Phi$ uses to attempt to compute a 1-generic.

We will also have labels of the form $B_{i,j}$, thought of as coding
labels. These labels will be placed on $2^{<\omega}$ in a $\Sigma^0_1$
way, although for simplicity of the construction, we will assume we
are able to place a countable computable set of $B$ labels on one node
during a single stage of the construction. They will be the labels
that the $\Psi_i$ use to attempt to compute $K$.

These will be the only labels that we place.

At each stage, $s$, of the construction, for each
$\sigma\in2^{<\omega}$ of length $<s$, we choose some new $n\in \N$,
and put $n$ into every set whose label is on an initial segment of
$\sigma$ at stage $s$. That number $n$ will not go into any other
sets. Thus, labeling $2^{<\omega}$ will be equivalent to defining (and
producing computations of) all of the sets $X_i$.

\vspace{5pt}

Throughout the construction, we will frequently wish to alternate
between objects that we construct and the sets of reals whose labels
appear on those objects. To this end, we make a collection of
notational definitions.

First, the labels will also be constant symbols for the reals that
they label. So, for instance, we begin the construction by placing the
label $A_\lambda$ onto the empty node $\lambda$. Since this is the
first label that we use, that means the label of $X_0$ will be
$A_\lambda$. This convention will allow us to treat $A_\lambda$ as a
symbol for a real (and in particular, have that $X_0=A_\lambda$).

If $f:\N\rightarrow\N$ is a function, then
$\widetilde f=\{X_i:i\in\ran(f)\}$.

A \emph{node} in $2^{<\omega}$ is an element of $2^{<\omega}$.

A \emph{finite path} in $2^{<\omega}$ is a sequence of successive
nodes in $2^{<\omega}$. If the finite path begins at the root of
$2^{<\omega}$, then it is the set of $\tau\preceq\sigma$ for some
$\sigma\in 2^{<\omega}$. This finite path will also sometimes be
denoted $\sigma$.

If $P\in 2^\omega$ is a real, regarded as an infinite path through
$2^{<\omega}$, then $\widetilde P$ is the set of reals whose labels
appear on initial segments of $P$.

If $\sigma\in 2^{<\omega}$, then $\widetilde\sigma$ is the set of
reals whose labels appear on proper initial segments of $\sigma$. This
set will also be referred to as the sequence of reals along the finite
path $\sigma$.

All of the functionals that we define in the construction will depend
only on the range of their oracle, $f$. Thus, the following definition
will be well defined. If $\theta$ is a Turing functional that we
define during this construction (so $\theta$ is either $\Phi$ or one
of the $\Psi_i$), and if $S$ is an infinite subset of the reals that
we construct during the construction, then $\theta^S=\theta^f$ where
$f$ is any function such that $\widetilde f=S$.

If $\sigma\in 2^{<\omega}$, and if $\theta$ is a functional that we
define during the construction, then $\theta^\sigma$ is the maximal
finite string $\alpha$ such that for any $S\supseteq\widetilde\sigma$,
$\theta^S\succ\alpha$. In other words, $\theta^\sigma$ is the initial
segment of the output of $\theta$ that is determined by $\sigma$. In
the construction, our functionals will be defined such that for any
$\sigma\in T$, $\theta^\sigma$ will always be defined.

\vspace{5pt}

Now that we have established this notation, note the following lemmas
that can be proved entirely from the above heuristics, without appeal
to the specific manner in which we decide how place our $B$ labels
onto $2^{<\omega}$.

\begin{lem}\label{L:finitepaths}

  Given any finite set, $S$, of reals from among the reals that we
  build, those reals have the finite intersection property if and only
  if there is a node $\sigma\in 2^{<\omega}$ such that
  $\widetilde\sigma\supseteq S$.

\end{lem}

\begin{proof}

  S is finite, and so, if there is such a $\sigma\in 2^{<\omega}$,
  then at some finite stage of the construction, the label of every
  real in $S$ has appears on an initial segment of $\sigma$. At that
  stage (and every stage thereafter), a common element is placed into
  all the reals in $\widetilde\sigma\supseteq S$.

  Conversely, if $S$ has the finite intersection property, then in
  particular, the reals in $S$ must have a common intersection
  (because $S$ is a finite set, and so a finite subset of
  itself). Thus, there must be some $n$ that is in all of the
  reals. That $n$ was chosen at some finite stage to go exactly into
  the reals along some $\sigma\in2^{<\omega}$. Thus, all the reals in
  our set must be in $\widetilde\sigma$.
\end{proof}

\begin{lem}

  Given any infinite path, $P$ through $2^{<\omega}$, $\widetilde P$
  has the finite intersection property.

\end{lem}

Note that this lemma does not claim that $\widetilde P$ is FIP with
respect to the sequence of all reals that we build.

\begin{proof}

  Given any finite number of reals in $\widetilde P$, there is a
  finite path that they are all on. By Lemma \ref{L:finitepaths},
  those reals have the finite intersection property, and so they
  intersect.
\end{proof}










\begin{lem}\label{L:containsapath}

  Let $S$ be an FIP set of reals for the sequence $\mathbb{X}$ that we
  build. Then there exists a unique infinite path $Y$ through
  $2^{<\omega}$ such that $\widetilde Y\subseteq S$.
\end{lem}

\begin{proof}
  First, we prove existence.

  Note first that if $\sigma$ and $\tau$ are incomparable in
  $2^{<\omega}$, then $A_\sigma$ and $A_\tau$ do not intersect,
  because those labels only appear on $\sigma$ and $\tau$
  respectively. Thus, for all $A$ sets that are in $S$, the subscripts
  of the labels must be comparable.

  Secondly, if $A_\sigma\in S$, and if $B_{i,j}\in\widetilde\sigma$,
  then $B_{i,j}\in S$. This is because any finite path containing
  $A_\sigma$ also contains $B_{i,j}$, and so by Lemma
  \ref{L:finitepaths}, if $S$ has the finite intersection property,
  and $A_\sigma\in S$, then $S\cup \{B_{i,j}\}$ has the finite
  intersection property. By maximality of $S$, we may conclude that
  $S=S\cup \{B_{i,j}\}$

  It remains to show that there exist infinitely many $\sigma$ such
  that $A_\sigma\in S$, because then $Y$ can be the union of those
  $\sigma$.

  \vspace{5pt}


  Assume this is not the case, and let $\sigma_0$ be maximal such that
  $A_{\sigma_0}\in S$. Let $\sigma_1={\sigma_0}^\smallfrown 0$, the
  left branch off of $\sigma_0$, and let
  $\sigma_2={\sigma_0}^\smallfrown 1$, the right branch off of
  $\sigma_0$. Because $S$ is maximal, and by choice of $\sigma_0$, we
  have that $S\cup \{A_{\sigma_1}\}$ does not have the finite
  intersection property. So fix a finite set $F_1\subseteq S$ such
  that $F_1\cup\{A_{\sigma_1}\}$ has an empty intersection. Likewise,
  fix a finite set $F_2\subseteq S$ such that
  $F_2\cup\{A_{\sigma_2}\}$ has an empty intersection.

  Let $F=F_1\cup F_2\cup\{A_{\sigma_0}\}$. Then $F$ is a finite subset
  of $S$, and so it has the finite intersection property, so by Lemma
  \ref{L:finitepaths}, there is some $\sigma$ such that
  $F\subseteq \widetilde\sigma$. By construction, $A_{\sigma_0}\in F$,
  and so $\sigma_0\preceq\sigma$. For $i=1,2$, $F\cup\{A_{\sigma_i}\}$
  does not have the finite intersection property, so by Lemma
  \ref{L:finitepaths}, $\sigma$ is incomparable to both $\sigma_1$ and
  $\sigma_2$. This is a contradiction, because every extension of
  $\sigma_0$ is comparable to either $\sigma_1$ or $\sigma_2$.
\end{proof}

We define $\Phi^S$ to be the $Y$ satisfying Lemma
\ref{L:containsapath}. Note that this is uniformly computable from any
$f$ such that $\widetilde f=S$ because, to compute the first $n$ bits
of $Y$, we need only wait to find some $\sigma$ of length $n$ such
that $A_\sigma$ is in $\widetilde f$.


\section{Construction}\label{fip2}

\noindent We now describe the way that the $B$ labels get placed onto
$2^{<\omega}$, and the way that our $\Psi$ functionals use them to
compute $K$ in the event that $\Phi^f$ is not 1-generic.

To make our notation simpler, we will assume that the c.e. sets $W_i$
always enumerate strings in $2^{<\omega}$. Under this convention, a
real $Y$ is 1-generic if and only if, for every $i$, either $W_i$
enumerates an initial segment of $Y$, or there is an initial segment
of $Y$ such that $W_i$ never enumerates any extensions of that initial
segment.

The construction is as follows. At stage $s$ of the construction, let
$K_s$ be the $s$th approximation to the halting set $K$. Then, for
each $i<s$, if $W_i$ enumerates $\sigma$ in $\leq s$ steps, and if
$W_i$ has not enumerated an initial segment of $\sigma$ at a previous
stage, then for every $j$ such that $j\notin K_s$, we place the label
$B_{i,j}$ onto the node $\sigma$. This completes the construction.

We now define $\Psi_i$ so that $\Psi_i^S(j)=1$ if $j\in K$, and
$\Psi_i^S(j)=0$ if $B_{i,j}\in S$.

More formally, to determine $\Psi_i^f(j)$, search for some value of
$s$ such that one of the following holds.

\begin{enumerate}
\item $j\in K_s$
\item There is some $t<s$ such that the labels for $X_0,\dots,X_t$
  have been determined by stage $s$ of the construction and

  $ (\exists r<s) (\exists q\leq t) \big(f(r)=q\ \ \ \& \ \ \text{the
    label of $X_q$ is $B_{i,j}$}\big).$

\end{enumerate}

If an $s$ satisfying (1) is found first, then halt and output 1. If an $s$ satisfying (2) is found first, then halt and output 0. (The primary reason for the formalism here is that the informal definition is not a definition, because both of the two ``if'' clauses could be true. In the conditions where $\Psi_i$ is supposed to work, exactly one of the two ``if'' clauses will be true, and the informal definition will be all we need.)\\

This completes the labeling of $2^{<\omega}$, and thus of $\mathbb X$,
and also $\Psi_i$ for each $i$. We now move on to prove that the
construction functions as desired.

\begin{lem}\label{L:main}

  Let $f$ be FIP for $\mathbb{X}$. Let $Y=\Phi^{f}$ and assume $Y$ is
  not 1-generic. Then there exists some $i$ such that $\Psi_{i}^f=K$.

\end{lem}

To explain how this will be proved, note the following lemma which
shows the only mechanism by which reals that are not in
$\widetilde{Y}$ might still be able to be in $\widetilde{f}$.

\begin{lem}\label{L:cofinal}

  Let $Y$ be a path through $2^{<\omega}$. Let $B_{i,j}$ be a $B$
  label that is not in $\widetilde{Y}$. Then
  $\widetilde{Y}\cup\{B_{i,j}\}$ has the finite intersection property
  if and only for every initial segment $\sigma$ of $Y$, there is an
  extension $\tau\succeq\sigma$ such that $B_{i,j}\in\widetilde\tau$.

\end{lem}

\begin{proof}

  Let $Y$ be an infinite path through $2^{<\omega}$. Let $B_{i,j}$ be
  a $B$ label that is not in $\widetilde{Y}$.

  Assume that $\widetilde{Y}\cup\{B_{i,j}\}$ has the finite
  intersection property, and let $\sigma\prec Y$. Then
  $A_\sigma\in\widetilde{Y}$ and so it must intersect $B_{i,j}$. By
  Lemma \ref{L:finitepaths}, there must be some $\tau$ such that
  $\{A_\sigma,B_{i,j}\}\subseteq\widetilde\tau$. Because
  $A_\sigma\in\widetilde\tau$, we have that $\tau\succeq\sigma$.

  Conversely, assume that for every $\sigma\prec Y$, there exists a
  $\tau\succ\sigma$ such that $B_{i,j}\in\widetilde\tau$. We must
  prove that $\widetilde{Y}\cup\{B_{i,j}\}$ has the finite
  intersection property. So let $F$ be a finite subset of
  $\widetilde{Y}\cup\{B_{i,j}\}$. Let
  $F_0=F\cap\widetilde{Y}$. Because $F_0$ is a finite subset of
  $\widetilde{Y}$, we have that there is a $\sigma\prec Y$ such that
  $F_0\subseteq\widetilde\sigma$. By assumption, extend $\sigma$ to a
  $\tau$ such that $B_{i,j}\in\widetilde\tau$. Then
  $F\subseteq\widetilde\tau$, and so, by Lemma \ref{L:finitepaths},
  $F$ has nonempty intersection.
\end{proof}

We now move on to prove Lemma \ref{L:main}

\begin{proof}

  Let $f$ be FIP for $\mathbb{X}$. Let $Y=\Phi^{f}$, the unique path
  through $2^{<\omega}$ such that
  $\widetilde Y\subseteq\widetilde{f}$. (Note that such a $Y$ is
  guaranteed to exist by Lemma \ref{L:containsapath}, and is equal to
  $\Phi^f$ by the comment after the proof of the lemma.)

  If $Y$ is not 1-generic, then $\widetilde Y\neq\widetilde f$, as
  follows. Let $i$ be such that $Y$ neither meets nor avoids
  $W_i$. Fix $j\notin K$. Then $B_{i,j}$ is not in
  $\widetilde{Y}$. This is because $Y$ does not meet $W_i$, and so we
  never put any $B_i$ labels onto initial segments of $Y$.

  Furthermore, every initial segment $\sigma$ of $Y$ has an extension
  $\tau\succeq\sigma$ such that $B_{i,j}\in\widetilde\tau$. This is
  because $Y$ does not avoid $W_i$, and so every initial segment of
  $Y$ has an extension $\tau$ that is in $W_i$, and every time such a
  $\tau$ enters $W_i$, we put the label $B_{i,j}$ onto that
  $\tau$. (Note, this uses the fact that $j\notin K$, and $K$ is c.e.,
  so we we have that $\forall s\ j\notin K_s$.)

  By Lemma \ref{L:cofinal}, $\widetilde Y\cup\{B_{i,j}\}$ has the
  finite intersection property, and so $\widetilde Y$ is not FIP,
  because it is not maximal. Thus, $\widetilde Y\neq\widetilde f$.

  \vspace{5pt}

  So now, let $B_{i_0,j_0}\in\widetilde f\setminus\widetilde Y$. Then
  we claim that $\{j:B_{i_0,j}\in\widetilde f\}={\bf a}r K$. (No
  $A$-labeled sets can be in $\widetilde f\setminus\widetilde Y$,
  because $A$-labels each occur only once, so if they are not in
  $\widetilde Y$, then there is an $A$-labeled node in $\widetilde Y$
  that they do not intersect.)

  To see this, first of all, note that $j_0\notin K$. Otherwise,
  $B_{i_0,j_0}$ only occurs at finitely many locations, and so by
  Lemma \ref{L:cofinal}, $\widetilde Y\cup\{B_{i_0,j_0}\}$ cannot have
  the finite intersection property.

  Note also that there is no $j$ such that $B_{i_0,j}\in\widetilde Y$.
  This is because $j_0\notin K$, and so, at any stage, if we put any
  $B_{i_0,j}$ onto any initial segment of $Y$, we would have placed
  $B_{i_0,j_0}$ onto that same location, contradicting our choice of
  $B_{i_0,j_0}$.

  If $j\notin K$, then $B_{i_0,j}$ occurs at exactly the locations
  where $B_{i_0,j_0}$ occurs, and so $B_{i_0,j}$ has finite
  intersection with exactly the sets that $B_{i_0,j_0}$ does, so if
  $B_{i_0,j_0}\in\widetilde f$, then $B_{i_0,j}\in\widetilde f$. (In
  this construction, $B_{i_0,j_0}=B_{i_0,j}$. If this is concerning to
  the reader, the construction could be modified so that we use the
  construction we present to add odd numbers to our sets, and we use
  the even numbers to give each set a single element that is in no
  other set. This does not change our FIP sequences, and ensures that
  our sets are all distinguished.)

  If $j\in K$, then $B_{i_0,j}$ occurs at finitely many locations, and
  so by Lemma \ref{L:cofinal}, $\widetilde Y\cup\{B_{i_0,j}\}$ does
  not have the finite intersection property, and so
  $B_{i_0,j}\notin\widetilde f$.

  Thus, $\{j:B_{i_0,j}\in\widetilde f\}={\bf a}r K$, and so
  $\Psi_{i_0}^f$ is a computation of $K$.
\end{proof}

We now prove the main result of this section.

\begin{thm}\label{T:fipcomputesgeneric}

  Let $X$ be a real that can compute an FIP function for any uniformly
  computable sequence of reals. Then $X$ can compute a Cohen
  1-generic.

\end{thm}

\begin{proof}

  The sequence of reals that we constructed is uniformly
  computable. Let $f$ be an FIP function for the sequence that we
  constructed. By Lemma \ref{L:main}, either $\Phi^f$ computes a
  1-generic, or there is some $i$ such that $\Psi_i^f$ computes
  $K$. We know that $K$ can be used to compute a 1-generic, and so it
  follows that in either case, $f$ can compute a 1-generic.

  Thus, if $X$ can compute such an $f$, then $X$ can compute a
  1-generic.
\end{proof}

\section{The 2IP Case}\label{2ip}
In this section we will modify the proof of the previous section to
work with the more delicate property of 2IP in place of FIP.  The
problems are caused by the $B$ sets from the previous
construction. Because they can appear multiple times on $2^{<\omega}$
in incomparable locations, it is possible for it there to be a finite
set involving $B$ sets (and potentially other sets as well) such that
any two of them appear together along some path, but such that they do
not all appear anywhere on $2^{<\omega}$ on a single finite path.

In this modification, we will produce a construction in which every
2IP set is, in fact, FIP. To do this, we will need to enforce
additional intersections among the sets: given any finite set with the
2-intersection property, we must have that finite set have a nonempty
intersection. We will accomplish this by adding a priority system to
the construction, allowing us greater control over which $B$ labels
occur in which locations.


The new construction is as follows.

We assume that for every $i$, $W_i$ never enumerates the empty node
$\lambda$. The set of indices of such $W_i$ is uniformly computable,
so this does not change the computability of the construction, and
every real trivially meets every $W_i$ that enumerates $\lambda$, so
meeting or avoiding those that do not is sufficient to be 1-generic.

We use the same $A$ labels as before.

The $B$ labels will now be indexed $B_{i,j,\nu}$, with $i,j\in\omega$,
and $\nu\in2^{<\omega}$. We use these labels in the same way that we
used $B_{i,j}$ previously, except that the $\nu$ is used to keep track
of where higher priority strategies have acted: At stage $s$, for each
$i<s$, in order, if $W_i$ enumerates $\sigma$ in $\leq s$ steps, then
let $\nu\preceq\sigma$ be the longest initial segment of $\sigma$ such
that $\nu$ has no extensions $\tau$ with a $B_{i'}$ label on it for
$i'<i$. If no such $\nu$ exists, then do nothing.

(Note, in particular, that this implies that lower priority labels
never get placed onto initial segments of nodes that already have
higher priority labels, except by the actions described in the next
paragraph.) Otherwise, if no initial segments of $\sigma$ except
possibly the empty string $\lambda$ have a $B_i$ label yet, then for
every $j$ such that $j\notin K_s$, we place the label $B_{i,j,\nu}$
onto the node $\sigma$.

In addition, however, when we ever place a label $B_{i_0,j_0,\nu_0}$
onto a node $\sigma$, for every $i<i_0$, for every $j$, and for every
$\nu\preceq\sigma$, we place the label $B_{i,j,\nu}$ on the empty
node, $\lambda$. (In essence, we discard any label we know we will
never use again.)


We build our reals $X_i$ as before, except using this differently
labeled tree.

This completes the construction.

We first prove that the modifications introduced achieve the desired
result: that at the end of the construction, every 2IP set is FIP.

\begin{lem}\label{L:2ipisfip}

  Given any finite set, $S$, of reals from among the reals that we
  build, those reals have the 2-intersection property if and only if
  there is a node $\sigma\in 2^{<\omega}$ such that
  $\widetilde\sigma\supseteq S$.

\end{lem}

From this, by Lemma \ref{L:finitepaths}, we may conclude that for this
set $\mathbb{X}$ that we build, a set is 2IP if and only if it is FIP.

\begin{proof}

  Note first that if there is a node $\sigma\in 2^{<\omega}$ such that
  $\widetilde\sigma\supseteq S$, then $S$ has the finite intersection
  property, and so $S$ has the 2-intersection property.

  We prove the converse by induction on the size of $S$, and
  secondarily by induction on the number of $B$-labeled sets in $S$.

  If $|S|=2$, then the lemma holds because the two sets in $S$
  intersect if and only if some element was placed into both of
  them. At the stage when that element was placed into both the sets,
  it must be because there was some path $\sigma\in 2^{<\omega}$ such
  that both those sets were in $\widetilde\sigma$.

  If $S$ has no $B$-labeled set, then the lemma holds because $A$
  labels still appear at most once each on $2^{<\omega}$. Thus, either
  there is some pair of them on incomparable nodes, and so $S$ does
  not have the 2-intersection property, or all of their nodes are
  comparable, and so one of those nodes is an extension of all the
  other nodes.

  If $S$ has one $B$-labeled set, then the lemma holds because that
  label must occur on a node that is comparable to the maximal node
  among those with $A$ labels (otherwise, the $S$ would not have the
  2-intersection property). Thus, that maximal node can be extended to
  (or is already an extension of) a path including a node labeled with
  that $B$-label.

  Otherwise, let $|S|>2$, and let $B_{i_0,j_0,\nu_0}$ and
  $B_{i_1,j_1,\nu_1}$ be two sets in $S$. We now consider three cases.

  \vspace{5pt}

  \noindent {\bf Case 1:} $i_0<i_1$, and at some point, label
  $B_{i_1,j_1,\nu_1}$ is placed on to $\lambda$.

  \vspace{5pt}

  Then, by induction, there is a $\sigma$ such that
  $\widetilde\sigma\supseteq (S\setminus\{B_{i_1,j_1,\nu_1}\})$. But
  then $\widetilde\sigma\supseteq S$, because $\lambda\preceq\sigma$.

  \vspace{5pt}

  \noindent {\bf Case 2:} $i_0<i_1$, and $B_{i_1,j_1,\nu_1}$ is never
  placed on to $\lambda$.

  \vspace{5pt}

  In this case, first note that because $S$ has 2IP,
  $B_{i_0,j_0,\nu_0}$ and $B_{i_1,j_1,\nu_1}$ have nonempty
  intersection, and so $B_{i_1,j_1,\nu_1}$ must at some point be
  placed on a node that is comparable to a node that already has the
  label $B_{i_0,j_0,\nu_0}$.

  Let $\sigma_0$ be a node with the label $B_{i_0,j_0,\nu_0}$ that at
  at some point has an extension with the label
  $B_{i_1,j_1,\nu_1}$. (If $B_{i_0,j_0,\nu_0}$ were on an extension of
  $B_{i_1,j_1,\nu_1}$, then $B_{i_1,j_1,\nu_1}$ would be placed on
  $\lambda$.)

  The we have that, $\nu_1\succeq\sigma_0$ because the $\nu$ value of
  a label is always the by definition of how $\nu$ values are chosen
  when $B$ labels are placed. For the same reason, the label
  $B_{i_1,j_1,\nu_1}$ is only ever placed on extensions of
  $\sigma_0$. (Instances of $B_{i_1}$ labels that are not on
  extensions of $\sigma_0$ would have a different $\nu$ value.) We
  proceed as in Case 1, but with $i_0$ and $i_1$ reversed:

  By induction, there is a $\sigma$ such that
  $\widetilde\sigma\supseteq (S\setminus\{B_{i_0,j_0,\nu_0}\})$. But
  then $\widetilde\sigma\supseteq S$, because $B_{i_1,j_1,\nu_1}$ is
  only ever placed on extensions of $\sigma_0$, and so
  $\sigma_0\preceq\sigma$. By choice of $\sigma_0$, $\sigma_0$ has the
  label $B_{i_0,j_0,\nu_0}$.

  \vspace{5pt}

  \noindent {\bf Case 3:} $i_0=i_1$.

  \vspace{5pt}

  If either $B_{i_0,j_0,\nu_0}$ or $B_{i_1,j_1,\nu_1}$ is ever placed
  onto $\lambda$, then we proceed as in Case 1.

  If $\nu_0=\nu_1$, then by symmetry, assume either that $j_0$ enters
  $K$ before $j_1$ or that neither $j_0$ nor $j_1$ is in $K$.
  Then every node with the label $B_{i_0,j_0,\nu_0}$ also has the
  label $B_{i_1,j_1,\nu_1}$. By induction, we may let $\sigma$ be such
  that $\widetilde\sigma\supseteq (S\setminus\{B_{i_1,j_1,\nu_1}\})$.
  $B_{i_0,j_0,\nu_0}$ is in $\widetilde\sigma$, and so
  $B_{i_1,j_1,\nu_1}$ is also in $\widetilde\sigma$.

  We now claim that it is not possible that $\nu_0\neq\nu_1$, and
  neither $B_{i_0,j_0,\nu_0}$ nor $B_{i_1,j_1,\nu_1}$ is ever placed
  onto $\lambda$.

  Let $\sigma_0$ be any path whose labels include $B_{i_0,j_0,\nu_0}$
  and $B_{i_1,j_1,\nu_1}$. Then $\nu_0$ and $\nu_1$ are each the
  respective shortest initial segments of $\sigma_0$ that had no
  higher priority extensions when the respective labels were
  placed. Thus in between when the two labels were placed, one of
  them, by symmetry $\nu_0$ must have had a higher priority $B$ label
  placed on an extension of it. But then $B_{i_0,j_0,\nu_0}$ is placed
  on $\lambda$, providing a contradiction.
\end{proof}

With this lemma, we have that for this new construction, the FIP sets
are exactly the 2IP sets, because the finite intersection property and
the 2-intersection property agree on finite sets.

The construction is defined as before, and we define $\Phi$ as before,
so every lemma from the previous two sections holds except possibly
Lemma \ref{L:main}, which is the only Lemma that references the
specifics of how the $B$ sets are placed, and also the only lemma that
references the specifics of how the $\Psi_i$ were defined. We now
define our analogues of the $\Psi_i$ that we will need for this
construction.

Let $\Psi_{i,\nu}$ be the functional such that that
$\Psi_{i,\nu}^S(j)=1$ if $j\in K$, and $\Psi_{i,\nu}^S(j)=0$ if
$B_{i,j,\nu}\in S$. Formalize this in the same way as for the $\Psi_i$
of the previous construction.

\begin{lem}\label{L:main2}

  Let $f$ be FIP (or equivalently 2IP) for $\mathbb{X}$. Let
  $Y=\Phi^{f}$ and assume $Y$ is not 1-generic. Then there exist some
  $i,\nu$ such that $\Psi_{i,\nu}^f$ is a computation K.

\end{lem}

As a reminder, in this construction, we assume that for every $i$,
$W_i$ never enumerates $\lambda$.

\begin{proof}

  Let $f$ be FIP for $\mathbb{X}$. Let $Y=\Phi^{f}$, the unique path
  through $2^{<\omega}$ such that
  $\widetilde Y\subseteq\widetilde{f}$.

  If $Y$ is not 1-generic, then let $i_1$ be minimal such that $Y$
  neither meets nor avoids $W_{i_1}$. Let $i_0$ be minimal such that
  there is no $B_{i_0}$ label on an initial segment of $Y$ besides
  $\lambda$, and such that every initial segment of $Y$ has an
  extension with a $B_{i_0}$ label. We claim that $i_0$ exists. To
  prove this, we show that if every $i<i_1$ does not satisfy the
  property that $i_0$ must satisfy, then $i_1$ satisfies that
  property:

  Assume that every $i<i_1$ does not satisfy this property, and let
  $\sigma_1$ be the shortest initial segment of $Y$ such that for
  every $i<i_1$
  there is no extension of $\sigma_1$ with a $B_i$ label. Then,
  because $Y$ neither meets nor avoids $W_{i_1}$, there are cofinally
  many locations along $Y$ where a $B_{i_1}$ label will want to be
  placed. If those locations are extensions of $\sigma_1$, then that
  label will actually be placed, because a $\nu$ value will be found,
  and so every initial segment of $Y$ has an extension with a
  $B_{i_1}$ label. Furthermore $Y$ does not meet $W_i$, and so no
  $B_{i_1}$ label will be placed on an initial segment of $Y$ besides
  $\lambda$.

  Therefore, $i_0$ exists. Let $\nu_0$ be the shortest initial segment
  of $Y$ such that for every $i<i_1$
  there is no extension of $\nu_0$ with a $B_i$ label. We now claim
  that $\Psi_{i_0,\nu_0}^f$ is a computation of $K$.

  To prove this, we first show that for every $j$,
  $B_{i_0,j,\nu_0}\notin\widetilde Y$. By choice of $i_0$, we know
  that $B_{i_0,j,\nu_0}$ does not appear on an initial segment of $Y$
  besides $\lambda$, and by choice of $\nu_0$, $B_{i_0,j,\nu_0}$ is
  never put on $\lambda$.

  Next we show that if $B_{i,j,\nu}\notin \widetilde{Y}$, then
  $B_{i,j,\nu}$ occurs on extensions of all initial segments of $Y$ if
  and only if $i=i_0$, $\nu=\nu_0$, and $j\notin K$.

  By assumption, every initial segment of $Y$ has an extension with a
  $B_{i_0}$ label. If that label was placed after the stage at which
  every $\sigma\prec\nu_0$ is seen to have an extension with a $B_i$
  label for some $i<i_0$, then that $B_{i_0}$ label has $\nu$ value
  $\nu_0$. If $\j\notin K$, then the label $B_{i_0,j,\nu_0}$ is placed
  at that location together with the aforementioned label. Thus, if
  $i=i_0$, $\nu=\nu_0$, and $j\notin K$, then $B_{i,j,\nu}$ occurs on
  extensions of all initial segments of $Y$.

  We prove the converse in cases. Assume
  $B_{i,j,\nu}\notin\widetilde{Y}$.

  If $i<i_0$, then by choice of $i_0$, $B_{i,j,\nu}$ does not occur on
  extensions of all initial segments of $Y$.

  If $j\in K$, then $B_{i,j,\nu}$ appears at only finitely many
  locations, and so $B_{i,j,\nu}$ does not occur on extensions of all
  initial segments of $Y$.

  If $i=i_0$, and $j\notin K$, and $\nu\neq\nu_0$, then after the
  stage at which every $\sigma\prec\nu_0$ is seen to have an extension
  with a $B_i$ label for some $i<i_0$, the only $\nu$ value used for
  $B_{i}$ labels is $\nu_0$, so $B_{i,j,\nu}$ does not occur on
  extensions of all initial segments of $Y$.

  Finally, if $i>i_0$, then first note that if $B_{i,j,\nu}$ does not
  occur on $\lambda$, then $B_{i,j,\nu}$ only occurs on extensions of
  $\nu$, and so if $B_{i,j,\nu}$ occurs on extensions of all initial
  segments of $Y$, then $\nu\prec Y$. So then, choose $\sigma\prec Y$
  such that $\sigma\succ\nu$. By assumption on $i_0$, at some stage, a
  $B_{i_0}$ label is placed on an extension of $\sigma$. At that
  stage, $B_{i,j,\nu}$ is placed on $\lambda$, and so
  $B_{i,j,\nu}\in\widetilde{Y}$, providing a contradiction.

  Therefore, we may conclude that
  $\widetilde{f}=\widetilde{Y}\cup\{B_{i_0,j,\nu_0}:j\notin K\}$ has
  the finite intersection property, and no other reals may be added to
  $\widetilde{Y}$ while preserving the finite intersection property.

  Thus $\Psi_{i_0,\nu_0}^f$ is a computation K.
\end{proof}

As before, we may now conclude our desired result.

\begin{thm}\label{T:2ipcomputesgeneric}

  Let $X$ be a real that can compute an 2IP function for any uniformly
  computable sequence of reals. Then $X$ can compute a Cohen
  1-generic.

\end{thm}

\section{Reverse Mathematics}

In this section, we discuss the reverse mathematical consequences of
our work from the previous sections. We discuss the reverse
mathematical principles FIP and 2IP, defined by Dzhafarov and Mummert
in \cite{DM13}, we define the principle ``1-GEN,'' and we use work
from this paper as well as from \cite{DDGT} to prove that FIP and
1-GEN are equivalent over RCA$_0$.
 
In \cite{DM13}, it was proved that FIP implies 2IP over RCA$_0$, and
our work from Section \ref{2ip} shows that 2IP also implies 1-GEN, and
therefore FIP. However, this implication appears to require
I$\Sigma_2$ for the proof. Our proof utilizes I$\Sigma_2$ in one
location, and B$\Sigma_2$ in two others.

\begin{defn}[Dzhafarov, Mummert \cite{DM13}]
  FIP is the principle of 2nd order arithmetic that says:

  ``Let $\mathbb{X}$ be a real, thought of as coding the sequence
  $\langle X_i\rangle$ of reals that are its columns. Assume at least
  one $X_i$ is nonempty. Then there exists an $f$ such that
  $\{X_{f(i)}\}$ is maximal with respect to the finite intersection
  property.''
\end{defn}

(Our original definition of an FIP real was over computable values of
$\mathbb{X}$. This statement can thus be thought of as the statement
``every $\mathbb{X}$, computable or not, has an $f$ that is FIP with
respect to $\mathbb{X}$.'')

\begin{defn}[Dzhafarov, Mummert \cite{DM13}]
  nIP is the same principle of 2nd order arithmetic, except with
  ``finite intersection property'' replaced with ``n-intersection
  property.''

\end{defn}

\begin{defn}
  1-GEN is the principle of 2nd order arithmetic that says:


  ``For every uniformly $\Sigma_1$ collection of sets, there is an $X$
  that meets or avoids every set in the collection.''

\end{defn}

Note that in practice, this says that for every $Y$, there exists a
$X$ such that $X$ is Cohen 1-generic over $X$, because the $\Sigma_1$
collection can be a universal $\Sigma_1^Y$ set, and because, given any
uniformly $\Sigma_1$ collection of sets, any $X$ that is Cohen
1-generic over the parameters in the definition would necessarily meet
or avoid ever set in the collection.

\medskip

We begin by demonstrating the equivalence between FIP and 1-GEN.

In \cite{DDGT}, it was proven that any Turing degree that can compute
a 1-generic is also an FIP degree. We summarize this proof, observe
that the proof relativizes and verify that the details can be carried
out in RCA$_0$ to prove that 1-GEN implies FIP over RCA$_0$.

\begin{prop}\label{1direction}
  RCA$_0\vdash$ 1-GEN $\rightarrow$ FIP.

\end{prop}

\begin{proof}[Proof (sketch)]
  Given a sequence of sets $\mathbb{X}=\langle X_i\rangle$, let $f$ be
  a generic for the forcing in which conditions are partial functions
  $g$, with finite domain, such that $\{X_{g(i)}\}$ has nonempty
  intersection.

  If $f$ is generic for this forcing, then $\{X_{f(i)}\}$ clearly has
  the finite intersection property. Furthermore, it is maximal by the
  following argument.

  If, for some $n$, $\{X_{f(i)}\}\cup\{X_{n}\}$ has the finite
  intersection property, then at every stage during the forcing, it
  would have been possible to extend $g$ to have $n$ in the range of
  $g$. In this case, by genericity of $f$, $n$ must be in the range of
  $f$.

  This argument is readily formalized in RCA$_0$.
\end{proof}

The converse to Proposition \ref{1direction} follows from the work in
this paper.

\begin{prop}
  RCA$_0\vdash$ FIP $\rightarrow$ 1-GEN.

\end{prop}

\begin{proof}[Proof (sketch)]
  The proofs in Sections \ref{fip1} and \ref{fip2} are mostly finitary
  in nature, and so readily formalized in RCA$_0$.  The infinitary
  arguments mostly follow from the definitions, and so require no
  induction. (For instance, a label appears on cofinally many
  extensions of the path precisely if a given c.e. set of nodes is
  dense along that path.)

  The only use of induction is in Lemma \ref{L:containsapath}, which
  states that an FIP sequence must contain $\lbrace A_\sigma\rbrace$
  for $\sigma$ in a unique infinite path through $2^{<\omega}$. This
  conclusion requires $\Sigma_0$ induction: an infinite path is a
  collection of nodes that are all comparable including one node at
  each level. We show that the nodes are all comparable directly, and
  we show that there is one node at each level because the nodes are
  downward closed and there cannot be a last node.

  This can all be formalized in RCA$_0$.
\end{proof}

In \cite{DM13}, it was proved that FIP implies 2IP over RCA$_0$. We
sketch the proof for completeness.

\begin{prop}\cite{DM13}

  For every $n>m$, RCA$_0\vdash$ nIP $\rightarrow$ mIP.

  For every $n$, RCA$_0\vdash$ FIP $\rightarrow$ nIP.

  In particular, RCA$_0\vdash$ FIP $\rightarrow$ 2IP.
\end{prop}

\begin{proof}[Proof (sketch)]
  We sketch the proof that RCA$_0\vdash$ FIP $\rightarrow$ 2IP.

  Given $\mathbb{X}=\langle X_i\rangle$, we create a new sequence
  $\widetilde{\mathbb{X}}=\langle \widetilde X_i\rangle$ where,
  whenever we see a finite set of the $X_i$ having the 2-intersection
  property, we take an element, and put that element into every
  $\widetilde X_i$ in the corresponding set of $\widetilde X_i$, and
  nowhere else. This ensures that the finite intersection property for
  $\widetilde{\mathbb{X}}$ is equivalent to the 2-intersection
  property for ${\mathbb{X}}$, and hence that any FIP sequence for
  $\widetilde{\mathbb{X}}$ is a 2IP sequence for ${\mathbb{X}}$.

  The proofs of the other statements are similar.
\end{proof}

From our construction, we are able to obtain a partial converse, but
our construction appears to require $\Sigma_2$ bounding.

\begin{prop}
  RCA$_0+I\Sigma_2\vdash$ 2IP $\rightarrow$ 1-GEN.
\end{prop}

\begin{proof}[Proof (sketch)]
  Lemma \ref{L:2ipisfip} is a proof that uses induction, and even
  after simplification appears to require $\Delta_2$ induction. By a
  result of Slaman \cite{S}, $\Delta_2$ induction is equivalent to
  $\Sigma_2$ bounding.  The statement in question is the statement
  that says that given any finite set of labels used in the second
  construction, those labels have the 2-intersection property if and
  only if they appear together on some finite path.

  A priori, a label being used in the construction is a $\Sigma_1$
  event, but we may rephrase the lemma to be quantified over all
  labels that could potentially be used, and only over sets of size 2
  or greater. In this case, the lemma would vacuously true for labels
  not used in the construction: they neither appear on paths, nor have
  the 2-intersection property with anything.

  Furthermore, the ``if'' part of the statement can be proven globally
  without induction: If the labels appear along a finite path, then
  they have an element in all of them, and so they have the
  2-intersection property.

  Thus, induction is only used to prove that if the labels have the
  2-intersection property then they appear together along some
  path. This is still a $\Sigma_1\rightarrow\Sigma_1$ statement, which
  is a $\Delta_2$ statement.

  \vspace{5pt}

  Lemma \ref{L:main2} also requires induction for two reasons, and
  this is where we use $\Sigma_2$ induction (or, more precisely,
  $L\Pi_2$, which is equivalent to $I\Sigma_2$. For background on
  induction principles, see Theorem 5.1 of \cite{H}).

  First we need a highest priority $i_0$ that acts infinitely often
  along extensions of initial segments of $Y$. Strategy $i$ acting
  infinitely often along extensions of initial segments of $Y$ is
  $\Pi_2$:
$$(\forall\sigma\prec Y)(\exists\tau\succ\sigma) \text{(some $B_i$ label appears on $\tau\ \&$ no $B_i$ label appears on $\sigma$.)}$$

Next we need to know that after some initial segment of $Y$, strategy
$i_0$ can stop changing its $\nu$ values. This apparently requires
$\Sigma_2$ bounding: if each $i<i_0$ acts on extensions of only
finitely many initial segments of $Y$, then we need to know that there
is some initial segment of $Y$ after which strategy $i_0$ will no
longer need to change its $\nu$ value on extensions of that segment,
allowing it to act infinitely often with the same labels and achieving
its infinitary requirement. This is an instance of $\Sigma_2$
bounding.
\end{proof}

A note on the proof:

It is tempting to attempt to remove the first instance of $\Sigma_2$
induction by letting $i_1$ be any $i$ such that $Y$ neither meets nor
avoids $W_i$, and then claiming that either some higher priority
strategy acts infinitely often or we can win with strategy
$i_1$. Unfortunately, this does not seem to work, because we need the
higher priority strategy to act infinitely often and also not to have
any even higher priority strategies act infinitely often.

Without $\Sigma_2$ induction, one could potentially imagine a cut below which every strategy acts finitely often, and above which every strategy acts infinitely often, but gets injured infinitely often by higher priority strategies that are also above that cut.\\

Synthesizing the results of this section, we obtain the following
implications.

\begin{thm}
  \

  RCA$_0\vdash$ FIP $\leftrightarrow$ 1-GEN.

  RCA$_0+I\Sigma_2$ proves the following are equivalent.
  \begin{itemize}[nolistsep]

  \item
    FIP

  \item
    1-GEN.

  \item
    nIP (for any value of $n$)
  \end{itemize}

\end{thm}
In particular, the below statements are all equivalent in $\omega$
models of RCA$_0$. We do not know whether these equivalences hold over
RCA$_0$ in general.

\begin{question}\label{q1}
  Does RCA$_0$ prove FIP $\leftrightarrow$ 2IP?

\end{question}

Indeed, we have an entire tower of open questions, although a positive
answer to Question \ref{q1} would resolve all of them simultaneously.

\begin{question}

  Do there exist $n,m$ such that RCA$_0$ proves nIP $\leftrightarrow$
  mIP?

  What about RCA$_0+B\Sigma_2$?

\end{question}

\end{document}